\documentclass[10pt,a4paper]{amsart}
\usepackage{amsmath,amsfonts,amsthm,amssymb,graphicx,mathtools}
\usepackage[stretch=10]{microtype}
\usepackage[unicode]{hyperref}
\usepackage[titletoc,title]{appendix}
\usepackage{xcolor}
\definecolor{dark-red}{rgb}{0.4,0.15,0.15}
\definecolor{dark-blue}{rgb}{0.15,0.15,0.4}
\definecolor{medium-blue}{rgb}{0,0,0.5}
\hypersetup{
	pdftitle={Standard Zero-Free Regions for Rankin--Selberg \texorpdfstring{$L$}{L}-Functions via Sieve Theory},
	pdfauthor={Peter Humphries},
	pdfnewwindow=true,   
	colorlinks, linkcolor={dark-red},
	citecolor={dark-blue}, urlcolor={medium-blue}
}
\makeatletter
\newcommand*{\defeq}{\mathrel{\rlap{%
                     \raisebox{0.3ex}{$\m@th\cdot$}}%
                     \raisebox{-0.3ex}{$\m@th\cdot$}}%
                     =}
\makeatother
\renewcommand\aa{\mathfrak{a}}
\newcommand\A{\mathbb{A}}
\renewcommand\C{\mathbb{C}}
\newcommand\GL{\mathrm{GL}}
\newcommand\OO{\mathcal{O}}
\newcommand\pp{\mathfrak{p}}
\newcommand\Q{\mathbb{Q}}
\newcommand\qq{\mathfrak{q}}
\newcommand\R{\mathbb{R}}
\def\e{\varepsilon}
\DeclareMathOperator{\ad}{ad}
\DeclareMathOperator*{\Res}{Res}
\numberwithin{equation}{section}
\newtheorem{theorem}[equation]{Theorem}
\newtheorem{corollary}[equation]{Corollary}
\newtheorem{lemma}[equation]{Lemma}

\newtheorem{thm}[subsection]{Theorem}
\newtheorem{lem}[subsection]{Lemma}

\theoremstyle{remark}
\newtheorem{remark}[equation]{Remark}
\newtheorem{rem}[equation]{Remark}
\makeatletter
\let\@wraptoccontribs\wraptoccontribs
\makeatother

\begin{document}

\title[Zero-Free Regions for Rankin--Selberg $L$-Functions via Sieve Theory]{Standard Zero-Free Regions for Rankin--Selberg $L$-Functions via Sieve Theory}

%\date{\today}

\author{Peter Humphries}
\address{Department of Mathematics, University College London, Gower Street, London WC1E 6BT}
\email{\href{mailto:pclhumphries@gmail.com}{pclhumphries@gmail.com}}

\contrib[with an appendix by]{Farrell Brumley}



\thanks{Research supported by the European Research Council grant agreement 670239.}

\begin{abstract}
We give a simple proof of a standard zero-free region in the $t$-aspect for the Rankin--Selberg $L$-function $L(s,\pi \times \widetilde{\pi})$ for any unitary cuspidal automorphic representation $\pi$ of $\GL_n(\A_F)$ that is tempered at every nonarchimedean place outside a set of Dirichlet density zero.
\end{abstract}

\maketitle

%\tableofcontents

\section{Introduction}

Let $F$ be a number field, let $n$ be a positive integer, and let $\pi$ be a unitary cuspidal automorphic representation of $\GL_n(\A_F)$ with $L$-function $L(s,\pi)$, with $\pi$ normalised such that its central character is trivial on the diagonally embedded copy of the positive reals. The proof of the prime number theorem due to de la Valle\'{e}-Poussin gives a zero-free region for the Riemann zeta function $\zeta(s)$ of the form
\[\sigma > 1 - \frac{c}{\log(|t| + 3)}\]
for $s = \sigma + it$, and this generalises to a zero-free region for $L(s,\pi)$ of the form
\begin{equation}\label{zerofreeeq}
\sigma \geq 1 - \frac{c}{(n[F : \Q])^4 \log(\qq(\pi)(|t| + 3))}
\end{equation}
for some absolute constant $c > 0$, where $\qq(\pi)$ is the analytic conductor of $\pi$ in the sense of \cite[Equation (5.7]{IK}, with the possible exception of a simple real-zero $\beta_{\pi} < 1$ when $\pi$ is self-dual. A proof of this is given in \cite[Theorem 5.10]{IK}; the method requires constructing an auxiliary $L$-function having a zero of higher order than the order of the pole at $s = 1$, then using an effective version of Landau's lemma \cite[Lemma 5.9]{IK}.

Now let $\pi'$ be a unitary cuspidal automorphic representation of $\GL_{n'}(\A_F)$, and consider the Rankin--Selberg $L$-function $L(s,\pi \times \pi')$. Via the Langlands--Shahidi method, this extends meromorphically to the entire complex plane with at most a simple pole at $s = 1$, with this pole occurring precisely when $\pi' \cong \widetilde{\pi}$. Moreover, this method shows that $L(s,\pi \times \pi')$ is nonvanishing in the closed right half-plane $\Re(s) \geq 1$ \cite[Theorem]{Sha}.

\begin{remark}
One can also obtain the nonvanishing of $L(s,\pi \times \pi')$ on the line $\Re(s) = 1$ via the Rankin--Selberg method. For $n = n'$ and $\pi' \ncong \widetilde{\pi}$, this is shown in \cite[Theorem 6.1]{Mor85}; the method of proof nonetheless is equally valid for $n \neq n'$ or $\pi' \cong \widetilde{\pi}$, noting in the latter case that $L(s,\pi \times \widetilde{\pi})$ has a simple pole at $s = 1$ (see also \cite[Equation (1.5)]{Sar}). Note, however, that the product of $L$-functions considered in \cite[Remark, p.~198]{Mor85} may \emph{not} be used to show the desired nonvanishing of $L(1 + it,\pi \times \pi')$, but merely the nonvanishing of $L(1,\pi \times \pi')$.
\end{remark}

Proving zero-free regions for $L(s,\pi \times \widetilde{\pi})$ akin to \eqref{zerofreeeq}, on the other hand, seems to be much more challenging. The method of de la Valle\'{e}-Poussin relies on the fact that the Rankin--Selberg convolutions $L(s,\pi \times \pi)$ and $L(s,\pi \times \widetilde{\pi})$ exist and extend meromorphically to $\C$ with at most a simple pole at $s = 1$. For $L(s,\pi \times \pi')$, the associated Rankin--Selberg convolutions have yet to be proved to have these properties, so as yet this method is inapplicable.

\begin{remark}
Note that in \cite[Exercise 4, p.~108]{IK}, it is claimed that one can use this method to prove a zero-free region similar to \eqref{zerofreeeq} when $\pi' \ncong \pi$ and $\pi' \ncong \widetilde{\pi}$; however, the hint to this exercise is invalid, as the Dirichlet coefficients of the logarithmic derivative of the auxiliary $L$-function suggested in this hint are real but not necessarily nonpositive. (In particular, as stated, \cite[Exercise 4, p.~108]{IK} would imply the nonexistence of Landau--Siegel zeroes upon taking $f$ to be a quadratic Dirichlet character and $g$ to be the trivial character.)
\end{remark}

\begin{remark}
When at least one of $\pi$ and $\pi'$ is self-dual, then this method \emph{can} be used to prove a zero-free region akin to \eqref{zerofreeeq}. When both $\pi$ and $\pi'$ are self-dual, this is proved by Moreno \cite[Theorem 3.3]{Mor85} (see also \cite[Equation (1.6)]{Sar}). When only one of $\pi$ and $\pi'$ is self-dual, such a zero-free region has been stated by various authors (in particular, see \cite[p.~619]{GeLa}, \cite[p.~92]{GLS}, and \cite[p.~1]{GoLi}); to the best of our knowledge, however, no proof of this claim has appeared in the literature. In the appendix to this article written by Farrell Brumley, a complete proof of this result is given.
\end{remark}

In \cite{GeLa}, Gelbart and Lapid generalise Sarnak's effectivisation of the Langlands--Shahidi method for $\zeta(s)$ \cite{Sar} to prove a zero-free region for $L(s,\pi \times \pi')$ of the form
\[\sigma \geq 1 - \frac{c_{\pi,\pi'}}{|t|^{N_{\pi,\pi'}}}\]
for some positive constants $c_{\pi,\pi'}, N_{\pi,\pi'}$ dependent on $\pi$ and $\pi'$, provided that $|t|$ is sufficiently large; their method applies not only to automorphic representations of $\GL_n(\A_F)$ but to more general reductive groups.

In \cite{Bru} and \cite[Appendix]{Lap}, Brumley proves a more explicit zero-free region for $L(s,\pi \times \pi')$ that is also valid in the analytic conductor aspect and not just the $t$-aspect. For $\pi' \ncong \widetilde{\pi}$, this is of the form
\[\sigma \geq 1 - c \left(\left(\qq(\pi) \qq(\pi')\right)^{2(n + n')} (|t| + 3)^{2nn' [F : \Q]}\right)^{-\frac{1}{2} + \frac{1}{2(n + n')} - \e},\]
together with the bound
\[L(s,\pi \times \pi') \gg_{\e} \left(\left(\qq(\pi) \qq(\pi')\right)^{2(n + n')} (|t| + 3)^{2nn' [F : \Q]}\right)^{-\frac{1}{2} + \frac{1}{2(n + n')} - \e}\]
for $s$ in this zero-free region, while for $\pi' \cong \widetilde{\pi}$, this is of the form
\begin{equation}\label{Brumleyselfdualzerofree}
\sigma \geq 1 - c\left(\qq(\pi)^{8n} (|t| + 3)^{2n^2 [F : \Q]}\right)^{-\frac{7}{8} + \frac{5}{8n} - \e},
\end{equation}
together with the bound
\begin{equation}\label{Brumleyselfduallowerbound}
L(s,\pi \times \widetilde{\pi}) \gg_{\e} \left(\qq(\pi)^{8n} (|t| + 3)^{2n^2 [F : \Q]}\right)^{-\frac{7}{8} + \frac{5}{8n} - \e}
\end{equation}
for $s$ in this zero-free region.

Recently, Goldfeld and Li \cite{GoLi} have given a strengthening in the $t$-aspect of a particular case of Brumley's result, namely the case $\pi' \cong \widetilde{\pi}$ subject to the restriction that $F = \Q$ and that $\pi$ is unramified and tempered at every nonarchimedean place outside a set of Dirichlet density zero. With these assumptions, they prove the lower bound
\begin{equation}\label{GoLilowerbound}
L(1 + it,\pi \times \widetilde{\pi}) \gg_{\pi} \frac{1}{(\log(|t| + 3))^3}
\end{equation}
for $|t| \geq 1$, which gives a zero-free region of the form
\begin{equation}\label{GoLizerofree}
\sigma \geq 1 - \frac{c_{\pi}}{(\log(|t| + 3))^5}
\end{equation}
for some positive constant $c_{\pi}$ dependent on $\pi$ provided that $|t| \geq 1$. Their proof, like that of Gelbart and Lapid \cite{GeLa}, makes use of Sarnak's effectivisation of the Langlands--Shahidi method; the chief difference is that, like Sarnak but unlike Gelbart and Lapid, they are able to use sieve theory to obtain a much stronger zero-free region. On the downside, the proof is extremely long and technical, and, being written in the classical language instead of the ad\`{e}lic language, any generalisation of their method to arbitrary number fields and allowing ramification of $\pi$ would be a challenging endeavour. (Indeed, the Langlands--Shahidi method, in practice, is rather inexplicit at ramified places, though see \cite{Hum} for explicit calculations for the case $n = 1$ and $F = \Q$, so that $\pi$ corresponds to a primitive Dirichlet character.)

In this article, we give a simple proof of the following.

\begin{theorem}\label{mainthm}
Let $\pi$ be a unitary cuspidal automorphic representation of $\GL_n(\A_F)$ that is tempered at every nonarchimedean place outside a set of Dirichlet density zero. Then there exists an absolute constant $c_{\pi}$ dependent on $\pi$ (and hence also on $n$ and $F$) such that $L(s,\pi \times \widetilde{\pi})$ has no zeroes in the region
\begin{equation}\label{mainzerofree}
\sigma \geq 1 - \frac{c_{\pi}}{\log(|t| + 3)}
\end{equation}
with $|t| \geq 1$. Furthermore, we have the bound
\begin{equation}\label{mainlowerbound}
L(s,\pi \times \widetilde{\pi}) \gg_{\pi} \frac{1}{\log(|t| + 3)}
\end{equation}
for $s$ in this region.
\end{theorem}

In particular, we improve the zero-free region \eqref{GoLizerofree} and lower bound \eqref{GoLilowerbound} of Goldfeld and Li to \eqref{mainzerofree} and \eqref{mainlowerbound} respectively while removing Goldfeld and Li's restriction that $F = \Q$ and that $\pi$ is unramified at every place. Nonetheless, we still require that $\pi$ be tempered at every nonarchimedean place outside a set of Dirichlet density zero; moreover, this zero-free region is only in the $t$-aspect, unlike Brumley's zero-free region in the analytic conductor aspect.

The proof of \hyperref[mainthm]{Theorem \ref*{mainthm}} shares some similarities with the method of de la Valle\'{e}-Poussin. Once again, one creates an auxiliary $L$-function, though this has a zero of equal order to the order of the pole at $s = 1$. While Landau's lemma cannot be used in this setting to obtain a standard zero-free region, one can instead use sieve theory. This approach is discussed in \cite[Section 3.8]{Tit} when $L(s,\pi \times \widetilde{\pi})$ is the Riemann zeta function, so that $F = \Q$ and $\pi$ is trivial, and this method can also be adapted to prove a standard zero-free region in the $q$-aspect for $L(s,\chi)$, where $\chi$ is a primitive Dirichlet character; cf.~\cite{BR,Hum}.

By slightly different means, we sketch how to prove a weaker version of \hyperref[mainthm]{Theorem \ref*{mainthm}}.

\begin{theorem}\label{secondthm}
Let $\pi$ be a unitary cuspidal automorphic representation of $\GL_n(\A_F)$ that is tempered at every nonarchimedean place outside a set of Dirichlet density zero. Then for $|t| \geq 1$, we have the bound
\begin{equation}\label{secondlowerbound}
L(1 + it,\pi \times \widetilde{\pi}) \gg_{\pi} \frac{1}{(\log(|t| + 3))^3},
\end{equation}
and so there exists an absolute constant $c_{\pi}$ dependent on $\pi$ such that $L(s,\pi \times \widetilde{\pi})$ has no zeroes in the region
\begin{equation}\label{secondzerofree}
\sigma \geq 1 - \frac{c_{\pi}}{(\log(|t| + 3))^5}.
\end{equation}
\end{theorem}

Though this is a weaker result than \hyperref[mainthm]{Theorem \ref*{mainthm}}, the method of proof is of particular interest; it is essentially a generalisation from $\GL_1(\A_{\Q})$ to $\GL_n(\A_F)$ of the method of Balasubramanian and Ramachandra \cite{BR}. It turns out that Brumley's method \cite{Bru} in proving \eqref{Brumleyselfduallowerbound} is a natural generalisation of \cite{BR} except that sieve theory is not used and so the resulting lower bounds for $L(1 + it,\pi \times \widetilde{\pi})$ are not nearly as strong.

\hyperref[secondthm]{Theorem \ref*{secondthm}} gives the same bounds as obtained by Goldfeld and Li, and this is no accident. Goldfeld and Li create an integral of an Eisenstein series and obtain upper bounds for this integral via the Maa\ss{}--Selberg relation together with upper bounds for $L(1 + it, \pi \times \widetilde{\pi})$ and $L'(1 + it, \pi \times \widetilde{\pi})$, while they use the Fourier expansion of the Eisenstein series together with sieve theory to find lower bounds for this integral. In the proof of \hyperref[secondthm]{Theorem \ref*{secondthm}}, we follow Brumley's method of studying a smoothed average of the Dirichlet coefficients of an auxiliary $L$-function. Upper bounds for this smoothed average are then obtained via Perron's inversion formula and Cauchy's residue theorem, in place of Goldfeld and Li's usage of the Maa\ss{}--Selberg relation, together with upper bounds for $L(1 + it, \pi \times \widetilde{\pi})$ and $L'(1 + it, \pi \times \widetilde{\pi})$; lower bounds for this smoothed average stem once again from sieve theory.

\section{Sieve Theory}

The $L$-function $L(s,\pi)$ of $\pi$ can be written as the Dirichlet series
\[L(s,\pi) = \sum_{\substack{\aa \subset \OO_F \\ \aa \neq \{0\}}} \frac{\lambda_{\pi}(\aa)}{N(\aa)^s}\]
for $\Re(s)$ sufficiently large, where $N(\aa) = N_{F/\Q}(\aa) \defeq \# \OO_F / \aa$, and extends to a meromorphic function on $\C$ with at most a simple pole at $s = 1$ if $n = 1$ and $\pi$ is trivial, so that $L(s,\pi) = \zeta_F(s)$. Similarly, the Rankin--Selberg $L$-function $L(s,\pi \times \widetilde{\pi})$ is meromorphic on $\C$ with only a simple pole at $s = 1$. We denote by $\Lambda_{\pi \times \widetilde{\pi}}(\aa)$ the coefficients of the Dirichlet series for $-\frac{L'}{L}(s,\pi \times \widetilde{\pi})$, so that
\[-\frac{L'}{L}(s,\pi \times \widetilde{\pi}) = \sum_{\substack{\aa \subset \OO_F \\ \aa \neq \{0\}}} \frac{\Lambda_{\pi \times \widetilde{\pi}}(\aa)}{N(\aa)^s}.\]
These coefficients are nonnegative; see \cite[Remark, p.~138]{IK}. Moreover, the residue of this at $s = 1$ is $1$, and we have that
\[\Lambda_{\pi \times \widetilde{\pi}}(\pp) = \left|\lambda_{\pi}(\pp)\right|^2 \log N(\pp)\]
whenever $\pi$ is unramified at $\pp$.

We denote by $S_{\pi}$ the set of places of $F$ at which $\pi$ is either ramified or nontempered.

\begin{lemma}[{\cite[Lemmata 12.12 and 12.15]{GoLi}}]\label{sievelemma}
Suppose that $\pi$ is tempered at every nonarchimedean place outside a set of Dirichlet density zero. For $Y \gg_{\pi} (|t| + 3)^2$,
\[\sum_{\substack{Y \leq N(\pp) \leq 2Y \\ \pp \notin S_{\pi}}} \left|\lambda_{\pi}(\pp) \right|^2 \left|1 + N(\pp)^{it}\right|^2 \gg_{\pi} \frac{Y}{\log Y}.\]
\end{lemma}

\begin{proof}
We use Ikehara's Tauberian theorem and the fact that $S_{\pi}$ has Dirichlet density zero to see that
\begin{equation}\label{Tauberian}
\sum_{\substack{Y \leq N(\pp) \leq 2Y \\ \pp \notin S_{\pi}}} \left|\lambda_{\pi}(\pp)\right|^2 \log N(\pp) = \sum_{Y \leq N(\aa) \leq 2Y} \Lambda_{\pi \times \widetilde{\pi}}(\aa) + o_{\pi}(Y) = Y + o_{\pi}(Y).
\end{equation}
The assumption that $\pi$ is tempered at every nonarchimedean place outside a set of Dirichlet density zero implies that $|\lambda_{\pi}(\pp)| \leq n$ whenever $\pp \notin S_{\pi}$, so that for any $C > 0$, the left-hand side of \eqref{Tauberian} is
\begin{multline*}
\sum_{\substack{Y \leq N(\pp) \leq 2Y \\ \pp \notin S_{\pi} \\ |\lambda_{\pi}(\pp)| < C}} \left|\lambda_{\pi}(\pp)\right|^2 \log N(\pp) + \sum_{\substack{Y \leq N(\pp) \leq 2Y \\ \pp \notin S_{\pi} \\ |\lambda_{\pi}(\pp)| \geq C}} \left|\lambda_{\pi}(\pp)\right|^2 \log N(\pp)	\\
\leq C^2 \sum_{Y \leq N(\pp) \leq 2Y} \log N(\pp) + n^2 \log 2Y \#\left\{Y \leq N(\pp) \leq 2Y : \pp \notin S_{\pi}, \ |\lambda_{\pi}(\pp)| \geq C \right\},
\end{multline*}
and as
\[\sum_{Y \leq N(\pp) \leq 2Y} \log N(\pp) = \sum_{Y \leq N(\aa) \leq 2Y} \Lambda(\aa) + o_F(Y) = Y + o_F(Y),\]
we ascertain that
\begin{equation}\label{densitylambdapiC}
\#\left\{Y \leq N(\pp) \leq 2Y : \pp \notin S_{\pi}, \ |\lambda_{\pi}(\pp)| \geq C \right\} \geq \frac{1 - C^2}{n^2} \frac{Y}{\log Y} + o_{\pi}\left(\frac{Y}{\log Y}\right).
\end{equation}

Next, for $C \in (0,2)$, we note that
\[\left|1 + N(\pp)^{it}\right| = 2\left|\sin\left(\frac{|t|}{2} \log N(\pp) - (2m - 1) \frac{\pi}{2}\right)\right|\]
for any integer $m$, and so via the bound $|\sin x| \leq |x|$, we have that
%and the fact that $[Y,2Y] \cap [a,b] \neq \emptyset$ iff both $a \leq 2Y$ and $b \geq Y$
\begin{multline}\label{<C}
\#\left\{Y \leq N(\pp) \leq 2Y : \left|1 + N(\pp)^{it}\right| < C \right\}	\\
\leq \sum_{\frac{|t|}{2\pi} \log Y - \frac{C}{2\pi} + \frac{1}{2} \leq m \leq \frac{|t|}{2\pi} \log 2Y + \frac{C}{2\pi} + \frac{1}{2}} \#\left\{e^{\frac{(2m - 1)\pi - C}{|t|}} \leq N(\pp) \leq e^{\frac{(2m - 1)\pi + C}{|t|}} \right\}.
\end{multline}
From \cite[Proposition 2]{GMP}, we have that
\[\pi_F(x + y) - \pi_F(x) \leq 4 [F : \Q] \frac{y}{\log y}\]
for $2 \leq y \leq x$, where $\pi_F(x) \defeq \#\{N(\pp) \leq x\}$; the proof of this reduces to the case $F = \Q$, in which case this is a well-known result that can be proven via the Selberg sieve (with the appearance of an additional error term) or the large sieve. So assuming that $\frac{1}{2\sqrt{Y}} \leq C \leq \frac{|t| \log 2}{2}$ and $Y > 4|t|^2$, the inner term on the right-hand side of \eqref{<C} is bounded by
\[64 [F : \Q] \frac{CY}{|t| \log \frac{Y}{4|t|^2}}\]
using the fact that $\log(e^u + 1) \geq \log u$ and $e^u - 1 \leq 2u$ for $u \in (0,1)$. Consequently,
\[\#\left\{Y \leq N(\pp) \leq 2Y : \left|1 + N(\pp)^{it}\right| < C \right\} \leq \frac{64 C [F : \Q] \log 2}{\pi} \frac{Y}{\log \frac{Y}{4|t|^2}}.\]
Since
\[\#\left\{Y \leq N(\pp) \leq 2Y : Y \notin S_{\pi}\right\} = \frac{Y}{\log Y} + o_F\left(\frac{Y}{\log Y}\right),\]
it follows that for $Y \gg_F (|t| + 3)^2$,
\begin{multline}\label{densitypitC}
\#\left\{Y \leq N(\pp) \leq 2Y : \pp \notin S_{\pi}, \ \left|1 + N(\pp)^{it}\right| \geq C \right\}	\\
\geq \left(1 - \frac{64 C [F : \Q] \log 2}{\pi}\right) \frac{Y}{\log Y} + o_{\pi}\left(\frac{Y}{\log Y}\right).
\end{multline}
By choosing $C$ sufficiently small in terms of $n$ and $F$, \eqref{densitylambdapiC} and \eqref{densitypitC} imply that
\[\#\left\{Y \leq N(\pp) \leq 2Y : \pp \notin S_{\pi}, \ \left|\lambda_{\pi}(\pp)\right| \left|1 + N(\pp)^{it}\right| \geq C^2 \right\} \gg_{\pi,C} \frac{Y}{\log Y},\]
from which the result follows.
\end{proof}

\begin{remark}
The only point at which we make use of the assumption that $\pi$ is tempered at every nonarchimedean place outside a set of Dirichlet density zero is in proving \eqref{densitylambdapiC}. It would be of interest whether an estimate akin to \eqref{densitylambdapiC} could be proved unconditionally.
\end{remark}

\begin{remark}
While the implicit constants in \hyperref[mainthm]{Theorems \ref*{mainthm}} and \ref{secondthm} depend on $\pi$, much of the argument still works if we keep track of this dependence in terms of the analytic conductor of $\pi$. The main issue seems to be the lower bound stemming from \hyperref[sievelemma]{Lemma \ref*{sievelemma}}; in particular, the use of Ikehara's Tauberian theorem to prove \eqref{Tauberian}. We could instead use \eqref{Brumleyselfduallowerbound} together with an upper bound for $L'(\sigma + it, \pi \times \widetilde{\pi})$ in the region \eqref{Brumleyselfdualzerofree} derived via the methods of Li \cite{Li} to prove \eqref{Tauberian} with an error term that is effective in terms of the analytic conductor of $\pi$, but the payoff would not be great as the weaker zero-free region \eqref{Brumleyselfdualzerofree} would only give a weak error term.
\end{remark}

\section{Proof of \texorpdfstring{\hyperref[mainthm]{Theorem \ref*{mainthm}}}{Theorem \ref{mainthm}}}
\label{proofofmainthmsect}

Let $\pi$ be a unitary cuspidal automorphic representations of $\GL_n(\A_F)$. Let $\rho = \beta + i\gamma$ be a nontrivial zero of $L(s,\pi \times \widetilde{\pi})$ with $1/2 \leq \beta < 1$ and $\gamma \neq 0$. We define
\[\Pi \defeq \pi \otimes |\det|^{\frac{i\gamma}{2}} \boxplus \pi \otimes |\det|^{-\frac{i\gamma}{2}}.\]
This is an isobaric (noncuspidal) automorphic representation of $\GL_{2n}(\A_F)$. The Rankin--Selberg $L$-function of $\Pi$ and $\widetilde{\Pi}$ factorises as
\begin{equation}\label{Lfactorisegamma}
L(s,\Pi \times \widetilde{\Pi}) = L(s, \pi \times \widetilde{\pi})^2 L(s + i\gamma, \pi \times \widetilde{\pi}) L(s - i\gamma, \pi \times \widetilde{\pi}).
\end{equation}
This is a meromorphic function on $\C$ with a double pole at $s = 1$, simple poles at $s = 1 \pm i\gamma$, and holomorphic elsewhere. We let $\Lambda_{\Pi \times \widetilde{\Pi}}(\aa)$ denote the coefficients of the Dirichlet series for $-\frac{L'}{L}(s,\Pi \times \widetilde{\Pi})$, so that
\[-\frac{L'}{L}(s,\Pi \times \widetilde{\Pi}) = \sum_{\substack{\aa \subset \OO_F \\ \aa \neq \{0\}}} \frac{\Lambda_{\Pi \times \widetilde{\Pi}}(\aa)}{N(\aa)^s}.\]
Again, these coefficients are nonnegative.

\begin{lemma}\label{L'Lupperlemma}
For $\sigma > 1$,
\[-\frac{L'}{L}(\sigma,\Pi \times \widetilde{\Pi}) < -\frac{2}{\sigma - \beta} + \frac{2}{\sigma - 1} + O\left(\log \qq(\Pi \times \widetilde{\Pi})\right).\]
\end{lemma}

\begin{proof}
By taking the real part of \cite[(5.28)]{IK}, we have that
\[-\frac{L'}{L}(\sigma + i\gamma,\pi \times \widetilde{\pi}) - \frac{L'}{L}(\sigma - i\gamma, \pi \times \widetilde{\pi}) < -\frac{2}{\sigma - \beta} + O\left(\log \qq(i\gamma,\pi \times \widetilde{\pi})\right)\]
for $\sigma > 1$; cf.~\cite[(5.37)]{IK}.
Similarly,
\[-2\frac{L'}{L}(\sigma,\pi \times \widetilde{\pi}) < \frac{2}{\sigma - 1} + O\left(\log \qq(\pi \times \widetilde{\pi})\right)\]
for $\sigma > 1$ via \cite[(5.37)]{IK}, using the fact that $\Lambda_{\pi \times \widetilde{\pi}}(\aa)$ is real.
\end{proof}

\begin{lemma}\label{unramgammalemma}
Suppose that $\pi$ is unramified at $\pp$. Then
\[\Lambda_{\Pi \times \widetilde{\Pi}}(\pp) = \left|\lambda_{\pi}(\pp)\right|^2 \left|1 + N(\pp)^{i\gamma}\right|^2 \log N(\pp).\]
\end{lemma}

\begin{proof}
Indeed, \eqref{Lfactorisegamma} implies that $\Lambda_{\Pi \times \widetilde{\Pi}}(\pp)$ is equal to
\[2\Lambda_{\pi \times \widetilde{\pi}}(\pp) + \Lambda_{\pi \times \widetilde{\pi}}(\pp) N(\pp)^{-i\gamma} + \Lambda_{\pi \times \widetilde{\pi}}(\pp) N(\pp)^{i\gamma} = \Lambda_{\pi \times \widetilde{\pi}}(\pp) \left|1 + N(\pp)^{i\gamma}\right|^2,\]
and $\Lambda_{\pi \times \widetilde{\pi}}(\pp) = \left|\lambda_{\pi}(\pp)\right|^2 \log N(\pp)$ whenever $\pi$ is unramified at $\pp$.
\end{proof}

\begin{corollary}\label{L'Llowercor}
Suppose that $\pi$ is tempered at every nonarchimedean place outside a set of Dirichlet density zero. Then for $\sigma > 1$,
\[-\frac{L'}{L}(\sigma,\Pi \times \widetilde{\Pi}) \gg_{\pi} \frac{(|\gamma| + 3)^{2(1 - \sigma)}}{\sigma - 1}.\]
\end{corollary}

\begin{proof}
We have that
\begin{align*}
-\frac{L'}{L}(\sigma,\Pi \times \widetilde{\Pi}) & \geq \sum_{\substack{N(\pp) \gg_{\pi} (|\gamma| + 3)^2 \\ \pp \notin S_{\pi}}} \frac{\log N(\pp)}{N(\pp)^{\sigma}} \left|\lambda_{\pi}(\pp)\right|^2 \left|1 + N(\pp)^{i\gamma}\right|^2	\\
& \gg_{\pi} \frac{(|\gamma| + 3)^{2(1 - \sigma)}}{\sigma - 1}
\end{align*}
by dividing into dyadic intervals and applying \hyperref[sievelemma]{Lemma \ref*{sievelemma}}.
\end{proof}

\begin{proof}[Proof of {\hyperref[mainthm]{Theorem \ref*{mainthm}}}]
By combining \hyperref[L'Lupperlemma]{Lemma \ref*{L'Lupperlemma}} and \hyperref[L'Llowercor]{Corollary \ref*{L'Llowercor}} and choosing $\sigma = 1 + c/\log(|\gamma| + 3)$, we find that
\[1 - \beta \gg_{\pi} \frac{1}{\log(|\gamma| + 3)},\]
which gives the zero-free region \eqref{mainzerofree}. Now using \cite[(5.28)]{IK}, we find in the region
\[\sigma \geq 1 - \frac{c_{\pi}}{2\log(|t| + 3)}\]
away from $t = 0$, we have that
\[-\frac{L'}{L}(s,\pi \times \widetilde{\pi}) \ll_{\pi} \log(|t| + 3).\]
Next, we note that
\[\log L(s, \pi \times \widetilde{\pi}) = \sum_{\substack{\aa \subset \OO_F \\ \aa \notin \{\{0\},\OO_F\}}} \frac{\Lambda_{\pi \times \widetilde{\pi}}(\aa)}{N(\aa)^s \log N(\aa)}\]
for $\Re(s) > 1$. So
\[\left|\log L(s, \pi \times \widetilde{\pi})\right| \leq \log L(\sigma, \pi \times \widetilde{\pi}).\]
Since $L(s,\pi \times \widetilde{\pi})$ has a simple pole at $s = 1$,
\[\log L(s, \pi \times \widetilde{\pi}) \ll_{\pi} \log \frac{1}{\sigma - 1}.\]
In particular, in the region
\[\sigma \geq 1 + \frac{1}{\log(|t| + 3)},\]
we have that
\[\log L(s, \pi \times \widetilde{\pi}) \ll_{\pi} \log \log(|t| + 3).\]
Now suppose that $s = \sigma + it$ with
\[1 - \frac{c_{\pi}}{2\log(|t| + 3)} \leq \sigma \leq 1 + \frac{1}{\log(|t| + 3)}.\]
Then $\log L(s, \pi \times \widetilde{\pi})$ is equal to
\[\log L\left(1 + \frac{1}{\log(|t| + 3)} + it, \pi \times \widetilde{\pi}\right) + \int_{1 + \frac{1}{\log(|t| + 3)} + it}^{s} \frac{L'}{L}(w,\pi \times \widetilde{\pi}) \, dw,\]
so again
\[\log L(s, \pi \times \widetilde{\pi}) \ll_{\pi} \log \log(|t| + 3).\]
Finally, we note that
\[\frac{1}{\left|L(s,\pi \times \widetilde{\pi})\right|} = \exp\left(-\Re\left(\log L(s,\pi \times \widetilde{\pi})\right)\right) \ll_{\pi} \log(|t| + 3),\]
which is equivalent to \eqref{mainlowerbound}.
\end{proof}

\begin{remark}
To prove \hyperref[mainthm]{Theorem \ref*{mainthm}} for $L(s,\pi \times \pi')$ with $\pi' \ncong \widetilde{\pi}$, we would need to replace \hyperref[sievelemma]{Lemma \ref*{sievelemma}} with a result of the form
\[\sum_{\substack{Y \leq N(\pp) \leq 2Y \\ \pp \notin S_{\pi} \cup S_{\pi'}}} \left|\lambda_{\pi}(\pp) + \lambda_{\widetilde{\pi}'}(\pp) N(\pp)^{i\gamma}\right|^2 \gg_{\pi,\pi'} \frac{Y}{\log Y},\]
but it is unclear how one might generalise the proof of \hyperref[sievelemma]{Lemma \ref*{sievelemma}} to obtain such a result.
\end{remark}

\section{Proof of \texorpdfstring{\hyperref[secondthm]{Theorem \ref*{secondthm}}}{Theorem \ref{secondthm}}}

For $t \in \R \setminus \{0\}$, define the isobaric automorphic representation $\Pi$ of $\GL_{2n}(\A_F)$ by
\[\Pi \defeq \pi \otimes |\det|^{\frac{it}{2}} \boxplus \pi \otimes |\det|^{-\frac{it}{2}}.\]
Then
\begin{equation}\label{Lfactoriset}
L(s,\Pi \times \widetilde{\Pi}) = L(s, \pi \times \widetilde{\pi})^2 L(s + it, \pi \times \widetilde{\pi}) L(s - it, \pi \times \widetilde{\pi}).
\end{equation}
This is a meromorphic function on $\C$ with a double pole at $s = 1$, simple poles at $s = 1 \pm it$, and holomorphic elsewhere.

We let $\lambda_{\Pi \times \widetilde{\Pi}}(\aa)$ denote the coefficients of the Dirichlet series for $L(s,\Pi \times \widetilde{\Pi})$, so that
\[L(s,\Pi \times \widetilde{\Pi}) = \sum_{\substack{\aa \subset \OO_F \\ \aa \neq \{0\}}} \frac{\lambda_{\Pi \times \widetilde{\Pi}}(\aa)}{N(\aa)^s}.\]
Again, the coefficients $\lambda_{\Pi \times \widetilde{\Pi}}(\aa)$ are nonnegative. Write
\[L(s,\Pi \times \widetilde{\Pi}) = \frac{r_{-2}}{(s - 1)^2} + \frac{r_{-1}}{s - 1} + O(1)\]
for $s$ near $1$ and
\[L(s,\Pi \times \widetilde{\Pi}) = \frac{r_{-1}^{\pm}}{s - (1 \pm it)} + O(1)\]
for $s$ near $1 \pm it$. Finally, we write
\[\zeta_F(s) = \frac{\gamma_{-1}(F)}{s - 1} + \gamma_0(F) + O(s - 1)\]
for $s$ near $1$.
% The constant $\gamma_{-1}(F)$ can be explicitly described via the analytic class number formula, while it is known that
%\[\gamma_0(F) = \gamma_{-1}(F) \lim_{x \to \infty} \left(\log x - \sum_{N(\aa) \leq x} \frac{\Lambda(\aa)}{N(\aa)}\right).\]

\begin{lemma}
We have that
\begin{align*}
\frac{r_{-2}}{\left|L(1 + it, \pi \times \widetilde{\pi})\right|^2} & = \gamma_{-1}(F)^2 L(1, \ad \pi)^2,	\\
\frac{r_{-1}}{\left|L(1 + it, \pi \times \widetilde{\pi})\right|^2} & = 2 L(1, \ad \pi) (\gamma_0(F) + \gamma_{-1}(F) L'(1, \ad \pi))	\\
& \qquad + 2 \gamma_{-1}(F)^2 L(1, \ad \pi)^2 \Re\left(\frac{L}{L}'(1 + it, \pi \times \widetilde{\pi})\right).
\end{align*}
Similarly,
\[r_{-1}^{\pm} = \gamma_{-1}(F) L(1, \ad \pi) L(1 \pm it, \pi \times \widetilde{\pi})^2 L(1 \pm 2it, \pi \times \widetilde{\pi}).\]
\end{lemma}

\begin{proof}
This follows from the factorisation $L(s,\pi \times \widetilde{\pi}) = \zeta_F(s) L(s, \ad \pi)$.
\end{proof}

\begin{lemma}[{\cite[Lemma 5.1]{GoLi}}]
We have that
\begin{align*}
L(1 + it, \pi \times \widetilde{\pi}) & \ll_{\pi} \log(|t| + 3),	\\
L'(1 + it, \pi \times \widetilde{\pi}) & \ll_{\pi} (\log(|t| + 3))^2.
\end{align*}
\end{lemma}

\begin{proof}
This is proved by Goldfeld and Li in \cite[Lemma 5.1]{GoLi} for $F = \Q$; the proof in this more generalised setting (via the approximate functional equation) follows mutatis mutandis.
\end{proof}

Together with the fact that $L(1,\ad \pi) \neq 0$, this shows the following.

\begin{corollary}\label{r2r1cor}
We have that
\begin{align*}
r_{-2} & \ll_{\pi} \left|L(1 + it, \pi \times \widetilde{\pi})\right| \log(|t| + 3),	\\
r_{-1} & \ll_{\pi} \left|L(1 + it, \pi \times \widetilde{\pi})\right| (\log(|t| + 3))^2,	\\
r_{-1}^{\pm} & \ll_{\pi} \left|L(1 + it, \pi \times \widetilde{\pi})\right| (\log(|t| + 3))^2.
\end{align*}
\end{corollary}

Now let $\psi \in C^{\infty}_c(0,\infty)$ be a fixed nonnegative function satisfying $\psi(x) = 1$ for $x \in [1,2]$ and $\psi(0) = 0$. The Mellin transform of $\psi$ is
\[\widehat{\psi}(s) \defeq \int_{0}^{\infty} \psi(x) x^s \, \frac{dx}{x},\]
which is entire with rapid decay in vertical strips. Define
\[F(Y) = \sum_{\substack{\aa \subset \OO_F \\ \aa \neq \{0\}}} \lambda_{\Pi \times \widetilde{\Pi}}(\aa) \psi\left(\frac{N(\aa)}{Y}\right).\]
We let $\qq(\Pi \times \widetilde{\Pi})$ denotes the analytic conductor of $\Pi \times \widetilde{\Pi}$; from \cite[(5.11)]{IK}, we have that
%noting that $\qq(\Pi \times \widetilde{\Pi}) = \qq(\pi \times \widetilde{\pi}) \qq(\pi' \times \widetilde{\pi}') \qq(it, \pi \times \pi') \qq(-it, \widetilde{\pi} \times \widetilde{\pi}')$
\begin{equation}\label{5.11}
\qq(\Pi \times \widetilde{\Pi}) \leq \qq(\pi)^{8n} (|t| + 3)^{2n^2 [F : \Q]}.
\end{equation}
On the other hand, it is easily seen that
\[\qq(\Pi \times \widetilde{\Pi}) \gg_{\pi} (|t| + 3)^2.\]

\begin{remark}
While \eqref{5.11} is stated in \cite[(5.11)]{IK}, a complete proof does not seem to have appeared in the literature. In the appendix to this article, a proof of (a more general version of) this statement is given.
\end{remark}

\begin{lemma}[Cf.~{\cite[Proof of Theorem 3]{Bru}}]\label{FPerronlemma}
For $Y \geq \qq(\Pi \times \widetilde{\Pi})$, there exists $\delta > 0$ such that
\begin{multline*}
F(Y) = r_{-2} \widehat{\psi}(1) Y \log Y + \left(r_{-1} \widehat{\psi}(1) + r_{-2} \widehat{\psi}'(1)\right) Y	\\
+ r_{-1}^{+} \widehat{\psi}(1 + it) Y^{1 + it} + r_{-1}^{-} \widehat{\psi}(1 - it) Y^{1 - it} + O(Y^{1 - \delta}).
\end{multline*}
\end{lemma}

\begin{proof}
Indeed,
\[F(Y) = \int_{\sigma - i\infty}^{\sigma + i\infty} L(s,\Pi \times \widetilde{\Pi}) \widehat{\psi}(s) Y^s \, ds\]
for $\sigma > 1$, and moving the contour to the left shows that $F(Y)$ is equal to
\[\left(\Res_{s = 1} + \Res_{s = 1 + it} + \Res_{s = 1 - it}\right) L(s,\Pi \times \widetilde{\Pi}) \widehat{\psi}(s) Y^s + \frac{1}{2\pi i} \int_{\sigma - i\infty}^{\sigma + i\infty} L(s,\Pi \times \widetilde{\Pi}) \widehat{\psi}(s) Y^s \, ds\]
for any $\sigma < 1$. The convexity bound of Li \cite{Li} and the rapid decay of $\widehat{\psi}$ in vertical strips imply that
\[\frac{1}{2\pi i} \int_{\sigma - i\infty}^{\sigma + i\infty} L(s,\Pi \times \widetilde{\Pi}) \widehat{\psi}(s) Y^s \, ds \ll_{\e} \qq(\Pi \times \widetilde{\Pi})^{\frac{1 - \sigma}{2} + \e} Y^{\sigma},\]
from which the result follows.
\end{proof}

In \cite{Bru} and \cite{Lap}, Brumley notes that
\[F(Y) \geq \sum_{Y \leq N(\aa^{2n}) \leq 2Y} \lambda_{\Pi \times \widetilde{\Pi}}(\aa^{2n})\]
and that $\lambda_{\Pi \times \widetilde{\Pi}}(\aa^{2n}) \geq 1$. This is paired with a modified version of \hyperref[FPerronlemma]{Lemma \ref*{FPerronlemma}} in order to prove effective lower bounds for $\left|L(1 + it, \pi \times \widetilde{\pi})\right|$ in terms of the analytic conductor $\qq(\pi \times \widetilde{\pi}, 1 + it)$. Instead of restricting the sum over integral ideals to those that are $2n$-powers and using the fact that $\lambda_{\Pi \times \widetilde{\Pi}}(\aa^{2n}) \geq 1$, our approach is to restrict to prime ideals at which $\pi$ is unramified and tempered and use sieve theory to show that $\lambda_{\Pi \times \widetilde{\Pi}}(\pp)$ is often not too small on dyadic intervals.

\begin{lemma}\label{unramifiedlemma}
Suppose that $\pi$ is unramified at $\pp$. Then
\[\lambda_{\Pi \times \widetilde{\Pi}}(\pp) = \left|\lambda_{\pi}(\pp)\right|^2 \left|1 + N(\pp)^{it}\right|^2.\]
\end{lemma}

\begin{proof}
This follows via the same method as the proof of \hyperref[unramgammalemma]{Lemma \ref*{unramgammalemma}}.
\end{proof}

\begin{proof}[Proof of {\hyperref[secondthm]{Theorem \ref*{secondthm}}}]
From \hyperref[FPerronlemma]{Lemma \ref*{FPerronlemma}} and \hyperref[r2r1cor]{Corollary \ref*{r2r1cor}}, we have that for $Y \geq \qq(\Pi \times \widetilde{\Pi})$, there exists $\delta > 0$ such that
\[F(Y) \ll_{\pi} \left|L(1 + it, \pi \times \widetilde{\pi})\right| Y (\log Y)^2 + Y^{1 - \delta}.\]
On the other hand, \hyperref[unramifiedlemma]{Lemmata \ref*{unramifiedlemma}} and \ref{sievelemma} imply that for $Y \gg_{\pi} (|t| + 3)^2$,
\[F(Y) \gg_{\pi} \frac{Y}{\log Y}.\]
This gives the lower bound \eqref{secondlowerbound}. Then as in \cite[Proof of Theorem 1.2]{GoLi}, \eqref{secondzerofree} follows via the mean value theorem.
\end{proof}

\section*{}
\vspace{-.6cm}

\subsection*{Acknowledgements}

The author would like to thank Peter Sarnak, Dorian Goldfeld, and Farrell Brumley for helpful discussions and comments.

\appendix

\section{Standard zero-free regions when at least one factor is self-dual, by Farrell Brumley}\footnote{LAGA - Institut Galil\'ee, 99 avenue Jean Baptiste Cl\'ement, 93430 Villetaneuse, France, \url{brumley@math.univ-paris13.fr}}\footnote{Supported by ANR grant 14-CE25}

The aim of this appendix is to provide a proof of the claim, stated in Gelbart--Lapid--Sarnak \cite[p.~92]{GLS} and Gelbart--Lapid \cite[p.~619]{GeLa}, that Rankin--Selberg $L$-functions are known to satisfy a standard zero-free region whenever at least one of the forms is self-dual. This is \hyperref[A1]{Theorem \ref*{A1}} below. The method is through the classical argument of de la Vall\'ee-Poussin. We take advantage of the occasion to clarify parts of the literature, and to justify, in \hyperref[BH-arch]{Lemma \ref*{BH-arch}}, another oft claimed inequality on the archimedean conductor.\footnote{I would like thank Philippe Michel and \'Etienne Fouvry for suggesting that I write up a proof of \hyperref[A1]{Theorem \ref*{A1}}. I am grateful as well to Peter Humphries for allowing me to include this appendix to his paper, and for suggesting many improvements to the proofs and exposition.}

\begin{thm}\label{A1}
Let $n,n'\geq 1$. Let $F$ be a number field. Let $\pi$ and $\pi'$ be unitary cuspidal automorphic representations of $\GL_{n}(\A_F)$ and $\GL_{n'}(\A_F)$, respectively. We normalize $\pi$ and $\pi'$ so that their central characters are trivial on the diagonally embedded copy of the positive reals. Assume that $\pi'$ is self-dual.

There is an effective absolute constant $c>0$ such that $L(s,\pi\times\pi')$ is non-vanishing for all $s=\sigma+it\in\C$ verifying
\[
\sigma\geq 1-\frac{c}{(n+n')^3\log \left( \qq(\pi)\qq(\pi')(|t|+3)^{n[F:\Q]}\right)},
\]
with the possible exception of one real zero whenever $\pi$ is also self-dual.
\end{thm}

\begin{rem}
In \cite{RW} it is shown that when $n=n'=2$, $L(s,\pi\times\pi')$ satisfies the conclusions of \hyperref[A1]{Theorem \ref*{A1}}, except possibly when it is divisible by the $L$-function of a quadratic character. A similar result holds for $n=n'=3$ when $\pi$ and $\pi'$ are symmetric square lifts from self-dual forms on $\GL_2$.
\end{rem}

\begin{rem} \hyperref[A1]{Theorem \ref*{A1}} implies a standard zero-free region for the $L$-function of a non-self dual cusp form $\pi$ on $\GL_n$, by taking $\pi'$ to be the trivial character on $\GL_1$. The fact that $L(s,\pi)$ admits no exceptional zeros whenever $\pi$ is not self-dual is originally due to Moreno \cite[Theorem 5.1]{Mor77} when $n=2$ and Hoffstein--Ramakrishnan \cite[Corollary 3.2]{HR95} for $n\geq 3$. (For complex characters it is classical.)
\end{rem}

\begin{rem}If $\pi$ is self-dual on $\GL_n$, then \hyperref[A1]{Theorem \ref*{A1}} allows for the possibility of a single {\it exceptional} zero, necessarily real, of $L(s,\pi)$. There are cases when this exceptional zero can be provably eliminated. To the best of our knowledge, these cases are, at the time of this writing, limited to the following situations:
\begin{enumerate}
\item\label{GL2} when $\pi$ is a cusp form on $\GL_2$, due to Hoffstein--Ramakrishnan \cite[Theorem C]{HR95};
\item\label{GL3} when $\pi$ is a cusp form on $\GL_3$. This is due to Banks \cite[Theorem 1]{Banks}, who verifies Hypothesis 6.1 in \cite{HR95}.
\item\label{GL5} when $\pi$ is a cusp form on $\GL_5$ which arises as the symmetric fourth power lift of a cusp form on $\GL_2$ which is not of solvable polynomial type. This is due to Ramakrishnan-Wang \cite{RW}; see the comments after Corollary C in {\it loc cit.};
\item\label{GL68} for the $L$-functions $L(s,\pi,{\rm sym}^6)$ and $L(s,\pi,{\rm sym}^8)$, when $\pi$ is a self-dual cusp from on $\GL_2$. This is Theorem D in \cite{RW}.
\end{enumerate}
All of these results build on the groundbreaking work of \cite{GHL}. Moreover, cases \eqref{GL5} and \eqref{GL68} make full use of the advances in functoriality by Kim and Shahidi \cite{KiSh,Kim}.

\end{rem}

\begin{rem}
We emphasize the importance of the cuspidality condition in \eqref{GL2} and \eqref{GL3} in the above remark, which rules out the divisibility of $L(s,\pi)$ by the $L$-function of a quadratic character.

For example, if $\pi$ is a dihedral form on $\GL_2$ over $F$, induced by a Hecke character $\chi$ of a quadratic field extension $E$, then $L(s,\pi)=L(s,\chi)$. Now if $\pi$ is cuspidal, $\chi$ does not factor through the norm, which (as was remarked in \cite{RW}) rules out $\chi$ real. One can then obtain a standard zero-free region for $\pi$ by appealing to the classical $\GL_1$ case for complex (Hecke) characters over $E$. The original argument given in \cite[Theorem B and Remark 4.3]{HR95} for dihedral forms on $\GL_2$ is, on the surface, more complicated, but this is simply due to to the more general framework in which it is set.

Similarly, the cuspidality condition for $\GL_3$ rules out the possibility that $\pi$ arises as the symmetric square lift of a dihedral form on $\GL_2$.
\end{rem}

\medskip

All of the above remarks pertain to results {\it in the full conductor aspect only}; for the $t$-aspect, we refer to the body of the paper.

\begin{proof}
For $t\in\R$ let
\[
\Pi=\pi\otimes |\det|^{it}\boxplus\widetilde{\pi}\otimes |\det|^{-it}\boxplus\pi'
\]
and $D(s)=L(s,\Pi\times\widetilde\Pi)$. Then we have the factorization $D(s)=L_1(s)L_2(s)$, where
\[
L_1(s)=L(s,\pi\times\widetilde{\pi})^2L(s,\pi'\times\pi')
\]
and
\[
L_2(s)=L(s+it,\pi\times\pi')^2 L(s-it,\widetilde{\pi}\times\pi')^2 L(s+2it,\pi\times\pi) L(s-2it,\widetilde{\pi}\times\widetilde{\pi}).
\]
Let $m\geq 1$ be the order of the pole of $D(s)$ at $s=1$. Then \cite[Theorem 5.9]{IK} (which is based on \cite[Lemma]{GHL}) states that there is a constant $\kappa>0$ such that $D(s)$ has no more than $m$ real zeros in the interval
\begin{equation}\label{range}
1-\frac{\kappa}{(n+n')^2(m+1)\log \qq(\Pi\times\widetilde{\Pi})}<\sigma<1.
\end{equation}

Let us calculate the integer $m$. The factor $L_1(s)$ has a pole of order $3$ at $s=1$. Moreover, if $t\neq 0$ the factor $L_2(s)$ is holomorphic at $s=1$. When $t=0$, the regularity of $L_2(s)$ at $s=1$ depends on whether or not $\pi$ is self-dual:
\begin{enumerate}
\item if $\pi$ is not self-dual and $t=0$, the function $L_2(s)$ is holomorphic at $s=1$, since necessarily $\pi\neq\pi'$ and $\pi\neq\widetilde{\pi}'$; 
\item if $\pi$ is self-dual and $t=0$, the function $L_2(s)$ has a pole of order $2$ or $4$, according to whether $\pi\neq\pi'$ or $\pi=\pi'$.
\end{enumerate}
We deduce that $m=3$ when either $\pi$ is not self-dual or $t\neq 0$. When $\pi$ is self-dual and $t=0$, we have $m=5$ or $7$, according to whether $\pi\neq\pi'$ or $\pi=\pi'$.

Now let $\sigma\in (0,1)$ and suppose that $L(s,\pi\times\pi')$ vanishes to order $r$ at $s=\sigma+it$. By the functional equation and the self-duality of $\pi'$, this is equivalent to $L(\sigma-it,\widetilde{\pi}\times\pi')$ vanishing to order $r$ at $s=\sigma-it$. From this it follows that $L_2(s)$, and hence $D(s)$, has a zero at $s=\sigma$ of order $4r$. Moreover, this is the case regardless of the value of $t$ or whether $\pi$ is self-dual. If $\sigma$ lies in the range \eqref{range}, then since $4r\leq m$, we deduce from the previous paragraph that $r=0$ whenever $\pi$ is not self-dual or $t\neq 0$, and $r\leq 1$ whenever $\pi$ is self-dual and $t=0$.

To finish the argument we must now majorize $\log\qq(\Pi\times\widetilde{\Pi})$. Corresponding to the factorization of $D(s)$ we have
\[
\qq(\Pi\times\widetilde{\Pi})=\qq(\pi\times\widetilde{\pi})^2\qq(\pi'\times\pi')\qq(it;\pi\times\pi')^2\qq(-it;\widetilde{\pi}\times\pi')^2\qq(2it;\pi\times\pi)\qq(-2it;\widetilde{\pi}\times\widetilde{\pi}).
\]
The bounds of \cite[Theorem 1]{BH}, applied to the finite conductor of each factor above, yield
\begin{equation}\label{BH-f}
\qq_f(\Pi\times\widetilde{\Pi})\leq (\qq_f(\pi)^2\qq_f(\pi'))^{4n+2n'}.
\end{equation}
For the archimedean conductor, \hyperref[BH-arch]{Lemma \ref*{BH-arch}} below implies that
\begin{equation}\label{upperRS}
\qq_\infty(\Pi\times\widetilde{\Pi})\leq C_1^{n+n'}(\qq_\infty(\pi)^2\qq_\infty(\pi'))^{4n+2n'}(1+|t|)^{(4nn'+2n^2)[F:\Q]},
\end{equation}
for an absolute constant $C_1>0$. This yields
\[
\log \qq(\Pi\times\widetilde{\Pi})\leq C_2(n+n') \log\left(\qq(\pi)\qq(\pi')(3+|t|)^{n[F:\Q]}\right)
\]
for an absolute constant $C_2>0$, which finishes the proof.\end{proof}

The following result -- the analog at archimedean places of the Bushnell--Henniart bounds \eqref{BH-f} on the Rankin--Selberg conductor -- has been claimed without proof in many sources, including \cite[(5.11)]{IK} and our own \cite{Bru} (to name just two). Nevertheless, there does not seem an available proof in the literature.

\begin{lem}\label{BH-arch}
Let $F_v$ be $\R$ or $\C$. Let $n,n'\geq 1$ be integers. Let $\pi_v$ and $\pi_v'$ be irreducible unitary generic representations of $\GL_n(F_v)$ and $\GL_{n'}(F_v)$, respectively. There is an absolute constant $C>0$ such that
\[
\qq_v(it;\pi\times\pi')\leq C^{n+n'}\qq_v(\pi)^{n'}\qq_v(\pi')^n(1+|t|)^{nn'[F_v:\R]}.
\]
\end{lem}

\begin{rem}
The constant $C$ in \hyperref[BH-arch]{Lemma \ref*{BH-arch}} can be explicitly computed, and the proof gives an exact value. But since the archimedean conductor should not be considered an ``exact quantity" (and conventions for the definition vary according to the source), it makes little sense to include the precise value of $C$ in the estimate.
\end{rem}

\begin{rem}
In the course of the proof, we shall recall the definition of the archimedean Rankin--Selberg and standard analytic conductor as given by Iwaniec--Sarnak in \cite{IS}. Their {\it ad hoc} recipe boils down to taking the product over all Gamma shifts arising in the local $L$-factors. It will be apparent that the definition of $\qq_v(it;\pi\times\pi')$ can be made in the admissible (rather than unitary generic) dual. One may drop the hypothesis of genericity (but not unitarity) in the statement of \hyperref[BH-arch]{Lemma \ref*{BH-arch}} at the price of allowing the constant $C$ to depend linearly on $n$ and $n'$.
\end{rem}

\begin{rem}\label{HR-remark}
In \cite[Lemma b]{HR95}, the authors prove something close to \hyperref[BH-arch]{Lemma \ref*{BH-arch}}, but their result only yields an upper bound of the form
\begin{equation}\label{HR-bound}
\qq_v(it;\pi\times\pi')\leq C(\qq_v(\pi)\qq_v(\pi')(1+|t|))^B,
\end{equation}
for some constants $B,C>0$, depending on $n$ and $n'$. Indeed, the archimedean factor of the ``thickened level" $M(\pi)$ introduced in [{\it loc. cit.}, Definition 1.4], is defined using the {\it sum}, rather than the product, of all Gamma shifts. (Note that in \cite{GHL} the max of the Gamma shifts is taken.) Thus $M_v(\pi)$, for $v\mid\infty$, behaves quantitatively much differently than the archimedean factor of the analytic conductor of Iwaniec--Sarnak \cite{IS}. Since its appearance, the latter has become the preferred measure of complexity in the study of $L$-functions.

It should be emphasized that since
\[
\log \qq_v(it;\pi\times\pi')\asymp \log M_v(it;\pi\times\pi')\quad\text{ and }\quad \log\qq_v(it;\pi)\asymp \log M_v(it;\pi),
\]
the bounds \eqref{upperRS} with {\it unspecified} exponents are consequences of the work of Hoffstein--Ramakrishnan. Thus the proof of \hyperref[A1]{Lemma \ref*{A1}} can be made to be independent of \hyperref[BH-arch]{Lemma \ref*{BH-arch}}, at the price of an inexplicit dependence in the implied constant on $n$ and $n'$.

In any case, the proof of \hyperref[BH-arch]{Lemma \ref*{BH-arch}} is closely modelled on that of \cite[Lemma b]{HR95}, with the appropriate modifications for dealing with analytic conductor. 
\end{rem}

\begin{proof} 
Using Langlands' classification of the admissible dual of $\GL_n(F_v)$ (see, for example, \cite{Kna94}), $\pi_v$ and $\pi_v'$ correspond to $\oplus\varphi_i$ and $\oplus\varphi_j'$, for irreducible representations $\varphi_i$ and $\varphi_j'$ of the Weil group $W_{F_v}$ of $F_v$. By definition, we have
\[
L(s,\pi_v)=\prod_i L(s,\varphi_i),\quad L(s,\pi_v')=\prod_j L(s,\varphi_j'),\quad L(s,\pi_v\times\pi_v')=\prod_{i,j}L(s,\varphi_i\otimes\varphi_j'),
\]
which gives rise to similar factorizations of the associated conductors. Let $d_i,d_j'$ denote the dimensions of $\varphi_i$ and $\varphi_j'$, respectively, so that $n=\sum d_i$ and $n'=\sum d_j'$. Dropping the indices $i$ and $j$, we must therefore prove that for irreducible representations $\varphi$ and $\varphi'$ of $W_{F_v}$, of respective dimensions $d$ and $d'$, we have
\begin{equation}\label{reduction}
\qq_v(it;\varphi\otimes\varphi')\leq C\qq_v(\varphi)^{d'}\qq_v(\varphi)^d(1+|t|)^{dd'[F_v:\R]},
\end{equation}
for an absolute constant $C>0$.

When $F_v=\C$, one has $W_\C=\C^\times$, so that all irreducible representations are one-dimensional. We may write any such character as $\chi_{k,\nu}(z)=(z/|z|)^k|z|^{2\nu}$, for $k\in\mathbb{Z}$ and $\nu\in\C$. Letting $\mu=\nu+|k|/2$, the associated $L$-factor (see \cite[(4.6)]{Kna94}) is $\Gamma_\C(s+\mu)$. The recipe of Iwaniec--Sarnak \cite[(21) and (31)]{IS} gives 
\[
\qq_v(it;\varphi)=(1+|it+\mu|)^2.
\]
Now if $\varphi=\chi_{k,\nu}$ and $\varphi'=\chi_{k',\nu'}$, then $\varphi\otimes\varphi'=\chi_{k+k',\nu+\nu'}$. This implies that 
\[
\qq_v(it;\varphi\otimes\varphi')=(1+|it+|k+k'|/2+\nu+\nu'|)^2.
\]
An application of the triangle inequality yields
\[
\qq_v(it;\varphi\otimes\varphi')\leq \left(1+|t|+\left(\frac{|k|}{2}+|\nu|\right)+\left(\frac{|k'|}{2}+|\nu'|\right)\right)^2.
\]
We claim that $\frac{|k|}{2}+|\nu|\ll |\frac{|k|}{2}+\nu|=|\mu|$. We may assume that $k\neq 0$. On one hand,
\[
\left(\frac{|k|}{2}+|\nu|\right)^2= \frac{k^2}{4}+|\nu|^2+|k||\nu|\leq \frac{k^2}{4}+|\nu|^2+\frac{k^2}{2}+\frac{|\nu|^2}{2}\leq 3\left(\frac{k^2}{4}+|\nu|^2\right).
\]
On the other, since $\pi$ and $\pi'$ are unitary generic, the Jacquet--Shalika bounds $|{\rm Re}(\nu)|\leq 1/2$ \cite[Corollary 2.5]{JS} (extended to the archimedean places by Rudnick--Sarnak in \cite[\S A.3]{RS}) imply
\[
\frac{k^2}{4}+|\nu|^2\leq |\mu|^2+\frac{|k|}{2}=|\mu|^2\left(1+\frac{|k|}{2|\mu|^2}\right)\leq |\mu|^2\left(1+\frac{2}{|k|}\right)\leq 3|\mu|^2.
\]
This proves the claim and implies $\qq_v(it;\varphi\otimes\varphi')\ll (1+|t|+|\mu|+|\mu'|)^2$. Using
\begin{equation}
\begin{aligned}\label{boring}
1+|t|+|\mu|+|\mu'|&\leq 1+|t|+|\mu|+|\mu'|+|\mu\mu'|\\
&=(1+|\mu|)(1+|\mu'|)+|t|\\
&\leq (1+|\mu|)(1+|\mu'|)(1+|t|),
\end{aligned}
\end{equation}
we establish \eqref{reduction} in the case $F_v=\C$.

When $F_v=\R$, each irreducible representation $\varphi$ of $W_\R=\C^\times\cup j\C^\times$ is of dimension 1 or 2. If $\varphi$ is one-dimensional, then its restriction to $\C^\times$ is $\chi_{0,\nu}$ for $\nu\in\C$ (see \cite[(3.2)]{Kna94}). We let $k=1-\varphi(j)\in\{0,2\}$. Writing $\mu=\nu+k/2$, we have $L(s,\varphi)=\Gamma_\R(s+\mu)$ and
\[
\qq_v(it,\varphi)=1+|it+\mu|.
\]
If $\varphi$ is two-dimensional, then it is the induction of $\chi_{k,\nu}$ from $\C^\times$ to $\GL_2(\R)$, where $k\geq 1$ and $\nu\in\C$. Putting $\mu=\nu+k/2$ we have $L(s,\varphi)=\Gamma_\C(s+\mu)$ and 
\[
\qq_v(it;\varphi)=(1+|it+\mu|)^2.
\]
In either case, let $(k,\nu)$ and $(k',\nu')$ be the parameters associated with $\varphi$ and $\varphi'$, respectively. We now examine the tensor products parameters.
\begin{enumerate}
\item If both $\varphi$ and $\varphi'$ are one-dimensional, then so is $\varphi\otimes\varphi'$, with parameter $(1-\varphi(j)\varphi'(j),\nu+\nu')$. Then \eqref{reduction} reads
\[
1+\bigg|it+\frac{1-\varphi(j)\varphi'(j)}{2}+\nu+\nu'\bigg|\leq C(1+|\mu|)(1+|\mu'|)(1+|t|).
\]
Applying the triangle inequality and $1-\varphi(j)\varphi'(j)\leq (1-\varphi(j))+(1-\varphi'(j))=k+k'$, the left-hand side is bounded above by
\[
1+|t|+\left(\frac{k}{2}+|\nu|\right)+\left(\frac{k'}{2}+|\nu'|\right).
\]
The same reasoning as in the complex case then establishes \eqref{reduction}.

\item If $\varphi$ is one-dimensional, and $\varphi'$ is irreducible and two-dimensional, then the twist $\varphi\otimes\varphi'$ is irreducible and two-dimensional, induced from $\C^\times$ by the character $\chi_{0,\nu}\chi_{k',\nu'}=\chi_{k',\nu+\nu'}$. Thus $\varphi\otimes\varphi'$ has parameters $(k',\nu+\nu')$. Inequality \eqref{reduction} is then equivalent to
\[
1+|it+\mu+\mu'|\leq C(1+|\mu|)(1+|\mu'|)(1+|t|).
\]
This follows (with $C=1$) from the triangle inequality and \eqref{boring}.

\item Suppose that $\varphi$ and $\varphi'$ are both irreducible and two-dimensional, and let $k\geq k'$. Then $\varphi\otimes\varphi'$ is the direct sum of two two-dimensional representations, induced from $\C^\times$ from the characters $\chi_{k,\nu}\chi_{k',\nu'}=\chi_{k+k',\nu+\nu'}$ and $\chi_{-k,-\nu}\chi_{k',\nu'}=\chi_{k'-k,\nu'-\nu}$. (Note that the latter representation is reducible when $k=k'$.) This shows that
\[
L(s,\varphi\otimes\varphi')=\Gamma_\C(s+\mu+\mu')\Gamma_\C(s+\mu-\mu')
\]
and
\[
\qq_v(it;\varphi\otimes\varphi')=(1+|\mu+\mu'|)^2(1+|\mu-\mu'|)^2.
\]
Then \eqref{reduction} is equivalent to 
\[
(1+|it+\mu+\mu'|)(1+|it+\mu-\mu'|)\leq C(1+|\mu|)^2(1+|\mu'|)^2(1+|t|)^2.
\]
This follows (with $C=1$) from applying the triangle inequality and \eqref{boring} to each factor on the left-hand side.
\end{enumerate}
This completes the proof of \hyperref[BH-arch]{Lemma \ref*{BH-arch}}.
\end{proof}


\begin{thebibliography}{GMP17}

\bibitem[BR76]{BR} R.~Balasubramanian and K.~Ramachandra, ``The Place of an Identity of Ramanujan in Prime Number Theory'', \textit{Proceedings of the Indian Academy of Sciences, Section A} \textbf{83}:4 (1976), 156--165. \textsc{doi}:\allowbreak\href{https://doi.org/10.1007/BF03051376}{10.1007/BF03051376}

\bibitem[Ban97]{Banks} William D.~Banks, ``Twisted Symmetric Square $L$-Functions and the Non-Existence of Siegel Zeros on $\GL(3)$'', \textit{Duke Mathematical Journal} \textbf{87}:2 (1997), 343--353. \textsc{doi}:\allowbreak\href{https://doi.org/10.1215/S0012-7094-97-08713-5}{10.1215/S0012-7094-97-08713-5}

\bibitem[BH97]{BH} C.~J.~Bushnell and G.~Henniart, ``An Upper Bound on Conductors for Pairs'', \textit{Journal of Number Theory} \textbf{65}:2 (1997), 183--196. \textsc{doi}:\allowbreak\href{https://doi.org/10.1006/jnth.1997.2142}{10.1006/jnth.1997.2142}

\bibitem[Bru06]{Bru} Farrell Brumley, ``Effective Multiplicity One on $\GL(n)$ and Zero-Free Regions of Rankin--Selberg $L$-Functions'', \textit{American Journal of Mathematics} \textbf{128}:6 (2006), 1455--1474. \textsc{doi}:\allowbreak\href{https://doi.org/10.1353/ajm.2006.0042}{10.1353/ajm.2006.0042}

\bibitem[GeLa06]{GeLa} Stephen S.~Gelbart and Erez M.~Lapid, ``Lower Bounds for $L$-Functions at the Edge of the Critical Strip'', \textit{American Journal of Mathematics} \textbf{128}:3 (2006), 619--638. \textsc{doi}:\allowbreak\href{https://doi.org/10.1353/ajm.2006.0024}{10.1353/ajm.2006.0024}

\bibitem[GLS04]{GLS} Stephen S.~Gelbart, Erez M.~Lapid, and Peter Sarnak, ``A New Method for Lower Bounds of $L$-Functions'', \textit{Comptes Rendus de l'Acad\'{e}mie des Sciences. S\'{e}rie 1, Math\'{e}matique} \textbf{339}:2 (2004), 91--94. \textsc{doi}:\allowbreak\href{https://doi.org/10.1016/j.crma.2004.04.024}{10.1016/j.crma.2004.04.024}

\bibitem[GHL94]{GHL} Dorian Goldfeld, Jeffrey Hoffstein and Daniel Lieman, ``An Effective Zero-Free Region'', appendix to ``Coefficients of Maass Forms and the Siegel Zero'' by Jeffrey Hoffstein and Paul Lockhart, \textit{Annals of Mathematics} \textbf{140}:1 (1994), 177--181. \textsc{doi}:\allowbreak\href{https://doi.org/10.2307/2118544}{10.2307/2118544}

\bibitem[GoLi17]{GoLi} Dorian Goldfeld and Xiaoqing Li, ``A Standard Zero Free Region for Rankin Selberg $L$-Functions'', \textit{International Mathematics Research Notices} rnx087 (2017), 70 pages. \textsc{doi}:\allowbreak\href{https://doi.org/10.1093/imrn/rnx087}{10.1093/imrn/rnx087}

\bibitem[GMP17]{GMP} Lo\"{i}c Greni\'{e}, Giuseppe Molteni, and Alberto Perelli, ``Primes and Prime Ideals in Short Intervals'', \textit{Mathematika} \textbf{63}:2 (2017), 364--371. \textsc{doi}:\allowbreak\href{https://doi.org/10.1112/S0025579316000310}{10.1112/S0025579316000310}

\bibitem[HR95]{HR95} Jeffrey Hoffstein and Dinakar Ramakrishnan, ``Siegel Zeros and Cusp Forms'', \textit{International Mathematics Research Notices} \textbf{1995}:6 (1995), 279--308. \textsc{doi}:\allowbreak\href{https://doi.org/10.1155/S1073792895000225}{10.1155/S1073792895000225}

\bibitem[Hum17]{Hum} Peter Humphries, ``Effective Lower Bounds for $L(1,\chi)$ via Eisenstein Series'', \textit{Pacific Journal of Mathematics} \textbf{288}:2 (2017), 355--375. \textsc{doi}:\allowbreak\href{https://doi.org/10.2140/pjm.2017.288.355}{10.2140/pjm.2017.288.355}

\bibitem[IK04]{IK} Henryk Iwaniec and Emmanuel Kowalski, \textit{Analytic Number Theory}, American Mathematical Society Colloquium Publications \textbf{53}, American Mathematical Society, Providence, 2004. \textsc{doi}:\allowbreak\href{https://doi.org/10.1090/coll/053}{10.1090/coll/053}

\bibitem[IS00]{IS} H.~Iwaniec and P.~Sarnak, ``Perspectives on the Analytic Theory of $L$-Functions'', in \textit{Vision in Mathematics. GAFA 2000 Special Volume, Part II}, editors N.~Alon, J.~Bourgain, A.~Connes, M.~Gromov, and V.~Milman, Birkh\"{a}user, Basel, 2000, 705--741. \textsc{doi}:\allowbreak\href{https://doi.org/10.1007/978-3-0346-0425-3_6}{10.1007/978-3-0346-0425-3\_6}

\bibitem[JS81]{JS} H.~Jacquet and J.~A.~Shalika, ``On Euler Products and the Classification of Automorphic Representations I'', \textit{American Journal of Mathematics} \textbf{103}:3 (1981), 499--558. \textsc{doi}:\allowbreak\href{https://doi.org/10.2307/2374103}{10.2307/2374103}

\bibitem[Kim03]{Kim}Henry H.~Kim, ``Functoriality for the Exterior Square of $\GL_4$ and the Symmetric Fourth of $\GL_2$'', with Appendix 1 by Dinakar Ramakrishnan and Appendix 2 by Henry H.~Kim and Peter Sarnak, \textit{Journal of the American Mathematical Society} \textbf{16}:1 (2003), 139--183. \textsc{doi}:\allowbreak\href{https://doi.org/10.1090/S0894-0347-02-00410-1}{10.1090/S0894-0347-02-00410-1}

\bibitem[KS02]{KiSh} Henry H.~Kim and Freydoon Shahidi, ``Cuspidality of Symmetric Powers with Applications", \textit{Duke Mathematical Journal} \textbf{112}:1 (2002), 177--197. \textsc{doi}:\allowbreak\href{https://doi.org/10.1215/S0012-9074-02-11215-0}{10.1215/S0012-9074-02-11215-0}

\bibitem[Kna94]{Kna94} A.~W.~Knapp, ``Local Langlands Correspondence: The Archimedean Case'', in \textit{Motives}, editors Uwe Jannsen, Steven L.~Kleiman, and Jean-Pierre Serre, Proceedings of Symposia in Pure Mathematics \textbf{55}:2, American Mathematical Society, Providence, 1994, 393--410. \textsc{doi}:\allowbreak\href{https://doi.org/10.1090/pspum/055.2}{10.1090/pspum/055.2}

\bibitem[Lap13]{Lap} Erez Lapid, ``On the Harish-Chandra Schwartz Space of $G(F) \backslash G(\A)$'', with an appendix by Farrell Brumley, in \textit{Automorphic Representations and $L$-Functions. Proceedings of the International Colloquium, Mumbai 2012}, editors D.~Prasad, C.~S.~Rajan, A.~Sankaranarayanan, and J.~Sengupta, Hindustan Book Agency, New Delhi, 2013, 335--377.

\bibitem[Li09]{Li} Xiannan Li, ``Upper Bounds for $L$-Functions at the Edge of the Critical Strip'', \textit{International Mathematics Research Notices} \textbf{2010}:4 (2010), 727--755. \textsc{doi}:\allowbreak\href{https://doi.org/10.1093/imrn/rnp148}{10.1093/imrn/rnp148}

\bibitem[Mor77]{Mor77} C.~J.~Moreno, ``Explicit Formulas in the Theory of Automorphic Forms'', in \textit{Number Theory Day}, editor Melvyn B.~Nathanson, Lecture Notes in Mathematics \textbf{626}, Springer--Verlag, Berlin, 1977, 73--216. \textsc{doi}:\allowbreak\href{https://doi.org/10.1007/BFb0063065}{10.1007/BFb0063065}

\bibitem[Mor85]{Mor85} Carlos J.~Moreno, ``Analytic Proof of the Strong Multiplicity One Theorem'', \textit{American Journal of Mathematics} \textbf{107}:1 (1985), 163--206. \textsc{doi}:\allowbreak\href{https://doi.org/10.2307/2374461}{10.2307/2374461}

\bibitem[RW03]{RW} Dinakar Ramakrishnan and Song Wang, ``On the Exceptional Zeros of Rankin--Selberg $L$-Functions'', \textit{Compositio Mathematica} \textbf{135}:2 (2003), 211--244. \textsc{doi}:\allowbreak\href{https://doi.org/10.1023/A:1021761421232}{10.1023/A:1021761421232}

\bibitem[RS96]{RS} Ze\'{e}v Rudnick and Peter Sarnak, ``Zeros of Principal $L$-Functions and Random Matrix Theory'', \textit{Duke Journal of Mathematics} \textbf{81}:2 (1996), 269--322. \textsc{doi}:\allowbreak\href{https://doi.org/10.1215/S0012-7094-96-08115-6}{10.1215/S0012-7094-96-08115-6}

\bibitem[Sar04]{Sar} Peter Sarnak, ``Nonvanishing of $L$-Functions on $\Re(s) = 1$'', in \textit{Contributions to Automorphic Forms, Geometry \& Number Theory}, editors Haruzo Hida, Dinakar Ramakrishnan, and Freydoon Shahidi, The John Hopkins University Press, Baltimore, 2004, 719--732.

\bibitem[Sha80]{Sha} Freydoon Shahidi, ``On Nonvanishing of $L$-Functions'', \textit{Bulletin of the American Mathematical Society} \textbf{2}:3 (1980), 462--464. \textsc{doi}:\allowbreak\href{https://doi.org/10.1090/S0273-0979-1980-14769-2}{10.1090/S0273-0979-1980-14769-2}

\bibitem[Tit86]{Tit} E.~C.~Titchmarsh, \textit{The Theory of the Riemann Zeta-function, Second Edition}, revised by D.~R.~Heath-Brown, Clarendon Press, Oxford, 1986.

\end{thebibliography}
\end{document}